\documentclass[11pt]{amsart}
\usepackage{setspace}
\usepackage{amsmath,amssymb}
\usepackage{amsthm}
\usepackage{fancybox}
\usepackage{algorithm, algpseudocode}
\usepackage{url}

\usepackage{color}
\usepackage{graphicx}
\usepackage{setspace}
\usepackage{comment}

\usepackage[top=1in,bottom=1in,left=1in,right=1in]{geometry}
 
\allowdisplaybreaks[4]
\usepackage{bbm}

\usepackage[initials]{amsrefs}
\usepackage{amsaddr}
\usepackage{mathtools} 
\mathtoolsset{showonlyrefs=true}

\usepackage{tabularx}

\usepackage[utf8]{inputenc}
\usepackage{hyperref}
\hypersetup{
    colorlinks=true,
    citecolor = blue,
    linkcolor=blue,
    filecolor=magenta,           
    urlcolor=cyan,
}

\theoremstyle{plain}
\newtheorem{thm}{Theorem}[section]
\newtheorem{lem}[thm]{Lemma}
\newtheorem{prop}[thm]{Proposition}

\newtheorem{cor}[thm]{Corollary}

\theoremstyle{definition}
\newtheorem{defn}[thm]{Definition}
\newtheorem{exmp}[thm]{Example}
\newtheorem{rmk}[thm]{Remark}

\usepackage{bm}

\newcommand{\matr}[1]{\bm{#1}}
\newcommand{\norm}[1]{\left\lVert#1\right\rVert}

\def\dim{\operatorname{dim}}
\def\dimA{\dim(\tensor{A})}

\def\poly{\operatorname{poly}}
\def\vec{\operatorname{Vec}}
\def\unfold{\operatorname{Unfold}}
\def\tensor#1{\mathcal{#1}}
\usepackage{stmaryrd}
\def\entry#1{\llbracket #1 \rrbracket}
\def\bigEntry#1{\Big\llbracket #1 \Big\rrbracket}
\def\onek{\bm{1}_{[k]}}
\def\zerok{\bm{0}_{[k]}}

\title[]{Operator Norm Inequalities between Tensor Unfoldings on the Partition Lattice}
\author{Miaoyan Wang$^1$, Khanh Dao Duc$^1$, Jonathan Fischer$^2$, and Yun S. Song$^{1,2,3}$}

\address{$^1$Department of Mathematics, University of Pennsylvania\\$^2$Department of Statistics, University of California, Berkeley\\
$^3$Computer Science Division, University of California, Berkeley }

\begin{document}


\begin{abstract}
Interest in higher-order tensors has recently surged in data-intensive fields, with a wide range of applications including image processing, blind source separation, community detection, and feature extraction.  A common paradigm in tensor-related algorithms advocates unfolding (or flattening) the tensor into a matrix and applying classical methods developed for matrices.  Despite the popularity of such techniques, how the functional properties of a tensor changes upon unfolding is currently not well understood.  
In contrast to the body of existing work which has focused almost exclusively on matricizations, we here consider all possible unfoldings of an order-$k$  tensor, which are in one-to-one correspondence with the set of partitions of $\{1,\ldots,k\}$.
We derive general inequalities between the $l^p$-norms of arbitrary unfoldings defined on the partition lattice.
In particular, we demonstrate how the spectral norm ($p=2$) of a tensor is bounded by that of its unfoldings, and obtain an improved upper bound on the ratio of the Frobenius norm to the spectral norm of an arbitrary tensor.  
For specially-structured tensors satisfying a generalized definition of orthogonal decomposability, we prove that the spectral norm remains invariant under specific subsets of unfolding operations.

\end{abstract}

\maketitle

\section{Introduction}

Tensors of order 3 or greater, known as higher-order tensors, have recently attracted increased attention in many fields across science and engineering. Methods built on tensors provide powerful tools to capture complex structures in data that lower-order methods may fail to exploit. Among numerous examples, tensors have been used to detect patterns in time-course data~\cite{zhao2005tricluster, sun2006beyond, omberg2007tensor, hoff2015multilinear} and to model higher-order cumulants~\cite{anandkumar2014tensor, mccullagh1984tensor, cardoso1990eigen}. However, tensor-based methods are fraught with challenges. Tensors are not simply matrices with more indices; rather, they are mathematical objects possessing multilinear algebraic properties. Indeed, extending familiar matrix concepts such as norms to tensors is non-trivial~\cite{lim2005singular, qi2005eigenvalues}, and computing these quantities has proven to be NP-hard~\cite{hillar2013most, friedland2014computational, friedland2016computational}. 

The spectral relations between a general tensor and its lower-order counterparts have yet to be studied.  There are generally two types of approaches underlying many existing tensor-based methods. The first approach flattens the tensor into a matrix and applies matrix-based techniques in downstream analyses, notably higher-order SVD \cite{kroonenberg1983three, de2000multilinear} and TensorFace~\cite{vasilescu2002multilinear}. Flattening is computationally convenient because of the ubiquity of well-established matrix-based methods, as well as the connection between tensor contraction and block matrix multiplication~\cite{ragnarsson2012block}. However, matricization leads to a potential loss of the structure found in the original tensor. This motivates the key question of how much information a flattening retains from its parent tensor. 

The second approach either handles the tensor directly or unfolds it into objects of order-3 or higher. Recent work on bounding the spectral norm of sub-Gaussian tensors reveals that solving for the convex relaxation of tensor rank by unfolding is suboptimal~\cite{tomioka2014spectral}. Interestingly, in the context of tensor completion, unfolding a higher-order tensor into a nearly cubic tensor requires smaller sample sizes than matricization~\cite{yuan2014tensor}. These results are probabilistic in nature and merely focus on a particular class of tensors. Assessing the general impact of unfolding operations on an arbitrary tensor and the role of the tensor's intrinsic structure remains challenging. 

The primary goal of this paper is to study the effect of unfolding operations on functional properties of tensors, where an unfolding is any lower-order representation of a tensor. We study the operator norm of a tensor viewed as a multilinear functional because this quantity is commonly used in both theory and applications, especially in tensor completion \cite{yuan2014tensor,mu2013square} and low-rank approximation problems \cite{tomioka2010estimation, yu2014approximate}. Given an order-$k$ tensor, we represent each possible unfolding operation using a partition $\pi$ of $[k]=\{1,\dots, k\}$, where 
 a block in $\pi$ corresponds to the set of modes that should be combined into a single mode.
Each unfolding is a rearrangement of the elements of the original tensor into a tensor of lower order.
Here we study the $l^p$ operator norms of all possible tensor unfoldings, which together define what we coin a ``norm landscape'' on the partition lattice.   A partial order relation between partitions enables us to find a path between an arbitrary pair of unfoldings and establish our main inequalities relating their operator norms. 
For specially-structured tensors satisfying a generalized definition of orthogonal decomposability, we show that the spectral norm ($p=2$) remains invariant under unfolding operations corresponding to a specific subset of partitions.  
To our knowledge, our results represent the first attempt to provide a full picture of the norm landscape over all possible tensor unfoldings.

The remainder of this paper is organized as follows.  In Section~\ref{sec:notation}, we introduce some notation, and relate the spectral norm and the general $l^p$-norm of a tensor.  We then describe in Section~\ref{sec:unfolding} general tensor unfoldings defined on the partition lattice.  In Section~\ref{sec:main_results}, we present our main results on the inequalities between the operator norms of any two tensor unfoldings and describe how the norm landscape changes over the partition lattice.  In Section~\ref{sec:OD}, we generalize the notion of orthogonal decomposable tensors and prove that the spectral norm is invariant within a specific set of tensor unfoldings.  We conclude in Section~\ref{sec:discussion} with discussions about our findings and avenues of future  work.

\section{Higher-order tensors and their operator norms}
\label{sec:notation}

An order-$k$ tensor $\tensor{A} = \entry{a_{i_1\dots\, i_k}}  \in \mathbb{F}^{d_1 \times \dots \times d_k}$ over a field $\mathbb{F}$ is a hypermatrix with dimensions $(d_1,\ldots,d_k)$ and entries $a_{i_1\dots i_k}\in\mathbb{F}$, for $1 \leq i_n \leq d_n$, $n=1,\ldots,k$.
In this paper, we focus on real tensors, $\mathbb{F}=\mathbb{R}$. 
The total dimension of $\tensor{A}$ is denoted by $\dimA =\prod_{n=1}^{k} d_n$. 

The vectorization of $\tensor{A}$, denoted $\vec(\tensor{A})$, is defined as the operation rearranging all elements of $\tensor{A}$ into a column vector. 
For ease of notation, we use the shorthand $[n]$ to denote the $n$-set $\{1,\dots, n\}$ for $n\in\mathbb{N}_+$, and sometimes write $\otimes^k_{n=1} \matr{x}_n=\matr{x}_{1}\otimes\cdots \otimes\matr{x}_{k}$ when space is limited. 
We use  $S^{d-1}=\{\matr{x}\in\mathbb{R}^d\colon \norm{\matr{x}}_2=1\}$ to denote the $(d-1)$-dimensional unit sphere, and $\matr{I}_d$ to denote the $d\times d$ identity matrix.

For any two tensors $\tensor{A}=\entry{a_{i_1\dots\, i_k}},$ $\tensor{B}=\entry{b_{i_1\dots\, i_k}} \in \mathbb{R}^{d_1 \times \cdots \times d_k}$ of identical order and dimensions, their inner product is defined as
\[
\langle \tensor{A},\ \tensor{B} \rangle = \sum_{i_1,\dots,i_k} a_{i_1\dots i_k} b_{i_1\dots i_k}, 
\]
while the tensor Frobenius norm of $\tensor{A}$ is defined as
\[
\| \tensor{A} \|_F = \sqrt{\langle \tensor{A},\ \tensor{A} \rangle} = \sqrt{\sum_{i_1, \dots, i_k} |a_{i_1 \dots i_k}|^2}, 
\]
both of which are analogues of standard definitions for vectors and matrices. 

Following \cite{lim2005singular}, we define the \emph{covariant multilinear matrix multiplication} of a tensor $\tensor{A}\in\mathbb{R}^{d_1\times\cdots\times d_k}$ by matrices $\matr{M}_1 = (m^{(1)}_{i_1j_1})\in \mathbb{R}^{d_1 \times s_1},\dots, \matr{M}_k=(m^{(k)}_{i_kj_k}) \in \mathbb{R}^{d_k \times s_k}$ as
\begin{equation} \label{eq:deffunctionalmat}
 \tensor{A}(\matr{M}_1, \dots, \matr{M}_k) = \bigEntry{\sum_{i_1=1}^{d_1}\cdots \sum_{i_k = 1}^{d_k} a_{i_1\dots i_k} m^{(1)}_{i_1j_1}\cdots  m^{(k)}_{i_kj_k}},\\
\end{equation}
which results in an order-$k$ tensor in $\mathbb{R}^{s_1\times\cdots\times s_k}$. This operation multiplies the $n^{th}$ mode of $\tensor{A}$ by the matrix $\matr{M}_n$ for all $n \in [k]$. Just as a matrix may be multiplied in up to two modes by matrices of consistent dimensions, an order-$k$ tensor can be multiplied by up to $k$ matrices in $k$ modes. In the case of $k=2$, $\tensor{A}$ is a matrix and $\tensor{A}(\matr{M}_1, \matr{M}_2) = \matr{M}_1^T \tensor{A} \matr{M}_2$. Sometimes we are interested in multiplying by vectors rather than matrices, in which case we obtain the $k$-multilinear functional $\tensor{A}\colon \mathbb{R}^{d_1} \times\cdots \times \mathbb{R}^{d_k} \to \mathbb{R}$ given by
\begin{align}\label{eq:deffunctional}
\tensor{A}(\matr{x}_1, \dots, \matr{x}_k)& = \sum_{i_1=1}^{d_1} \cdots \sum_{i_k = 1}^{d_k} a_{i_1\dots i_k} x^{(1)}_{i_1}\cdots  x^{(k)}_{i_k} \notag \\
&=\langle \tensor{A},\ \matr{x}_1\otimes \cdots \otimes \matr{x}_k \rangle,
\end{align}
where $\matr{x}_n=(x^{(n)}_1,\dots,x^{(n)}_{d_n})^T\in\mathbb{R}^{d_n}$, $n \in [k]$. Note that multiplying by a vector in $r$ modes in the manner defined in~\eqref{eq:deffunctional} reduces the order of the output tensor to $k-r$, whereas multiplying by matrices leaves the order unchanged.  Although the coordinate representation of a tensor as a hypermatrix provides a concrete description, viewing it instead as a multilinear functional provides a coordinate-free, basis-independent perspective which allows us to better characterize the spectral relations among different tensor unfoldings.

We define the \emph{operator norm}, or \emph{induced norm}, of a tensor $\tensor{A}$ using the associated $k$-multilinear functional~\eqref{eq:deffunctional}.

\begin{defn} [Lim~\cite{lim2005singular}]\label{def:norm}
Let $\tensor{A}\in\mathbb{R}^{d_1\times\cdots\times d_k}$ be an order-$k$ tensor. For any $1\leq p\leq \infty$, the $l^p$-norm of the multilinear functional associated with $\tensor{A}$ is defined as
\begin{align}\label{eq:norm}
\norm{\tensor{A}}_p &=\sup \left \{ {\tensor{A}(\matr{x}_1, \dots, \matr{x}_k) \over \norm{\matr{x}_1}_p \cdots \norm{\matr{x}_k}_p}\colon \matr{x}_n\neq 0, \matr{x}_n\in\mathbb{R}^{d_n}, n\in [k]   \right\} \notag \\
&= \sup \left\{ \tensor{A}(\matr{x}_1, \dots, \matr{x}_k)\colon \norm{\matr{x}_n}_p = 1, \matr{x}_n\in\mathbb{R}^{d_n}, n\in [k] \right\},
\end{align}
where $\norm{ \matr{x}_n }_p$ denotes the vector $l^p$-norm of $\matr{x}_n$. 
\end{defn}

\begin{rmk}
The special case of $p=2$ is called the \emph{spectral norm}, frequently denoted $\norm{ \tensor{A} }_\sigma$. By \eqref{eq:norm}, $\norm{\tensor{A}}_\sigma$ is the maximum value obtained as the inner product of the tensor $\tensor{A}$ with a rank-1 tensor, $\matr{x}_1\otimes \cdots \otimes \matr{x}_n$, of Frobenius norm 1 and of the same dimensions. This point of view provides an equivalent definition of $\norm{\tensor{A}}_\sigma$ as determining the best rank-1 tensor approximation to $\tensor{A}$, and we note that the rank-1 constraint becomes weaker as more unfolding is applied. See Section~\ref{sec:main_results} for further details. 
\end{rmk}

\medskip
Because we restrict all entries of $\tensor{A}$ and $\{ \matr{x}_n \}_{n=1}^k$ to be real, we need not take the absolute value of $\tensor{A}(\matr{x}_1, \dots, \matr{x}_k)$ as in \cite{lim2005singular}.  It is worth mentioning that the notion of tensor $l^p$-norms defined by~\eqref{eq:norm} are \emph{not} extensions of the classical matrix $l^p$-norms when $p\neq 2$. To see this, recall that for an $m \times n$ matrix $\matr{A}$, one usually defines the $l^p$ operator norm as 
\begin{equation}\label{eq:oldnorm}
\norm{\matr{A}}_p=\sup\left\{ {\norm{\matr{Ax}}_p \over \norm{\matr{x}}_p} \colon \matr{x}\neq \matr{0},\matr{x}\in\mathbb{R}^n  \right\}.
\end{equation}
In general~\eqref{eq:norm} and~\eqref{eq:oldnorm} are not equal even for matrices, illustrated by the following example:

\medskip
\begin{exmp} 
Let $\matr{A}=\entry{a_{ij}}$ be the $2 \times 2$ matrix
$
\matr{A}=
\begin{bmatrix}
1 &1 \\ 0 &4
\end{bmatrix}$
and consider $p=1$. Solving~\eqref{eq:oldnorm}, we have
\[
\norm{\matr{A}}_{1}=\sup \left\{ {\norm{\matr{Ax}}_{1}\over {\norm{\matr{x}}_{1}}} \colon \matr{x}\neq \matr{0}, \matr{x}\in\mathbb{R}^2  \right\}=\sup_{j} \sum_{i} |a_{ij}|=5.
\]
However, instead using~\eqref{eq:norm} gives
\begin{align}
\displaystyle \norm{\matr{A}}_1&=\sup \left\{ {\matr{x}_1^T\matr{A}\matr{x}_2\over \norm{\matr{x}_1}_1 \norm{\matr{x}_2}_1}\colon \matr{x}_n\in\mathbb{R}^2, \matr{x}_n\neq \matr{0},  n \in [2] \right\}\\
&=\sup\left\{ \matr{x}_1^T\matr{A}\matr{x}_2\colon \matr{x}_n\in\mathbb{R}^2, \norm{\matr{x}_n}_1=1,  n\in[2] \right\}\\
&=4,
\end{align}
which is neither the classical matrix $l^1$-norm (equal to 5), nor the entry-wise $l^1$-norm (equal to $\sum_{i,j}|a_{ij}|=6$).
\end{exmp}

\medskip
Throughout this paper, we adopt Definition~\ref{def:norm} and always use $\norm{ \ \cdot \ }_p$ to denote the $l^p$-norm defined therein, even for matrices. In fact,~\eqref{eq:oldnorm} defines an operator norm by viewing the matrix as a linear operator from $\mathbb{R}^{d_2}$ to  $\mathbb{R}^{d_1}$, whereas~\eqref{eq:norm} defines an operator norm in which the matrix defines a bilinear functional from $\mathbb{R}^{d_1} \times \mathbb{R}^{d_2}$ to $\mathbb{R}$. These two definitions are equivalent when $p=2$, but otherwise represent two different operators and result in two distinct operator norms. To be consistent with our treatment of tensors as $k$-multilinear functionals, we formulate matrices as bilinear functionals. 

For a given tensor, its $l^p$-norm and $l^q$-norm mutually control each other, and the comparison bound is polynomial in the total dimension of the tensor, $\dimA= \prod_{n=1}^{k} d_n$.\\

\begin{prop}[$l^p$-norm vs. $l^q$-norm] \label{thm:pnorm}
Let $\tensor{A}\in\mathbb{R}^{d_1\times \cdots \times d_k}$ be an order-$k$ tensor. Suppose $\norm{\cdot}_p$ and $\norm{\cdot}_q$ are two norms defined in~\eqref{eq:norm} with $q\geq p\geq 1$. Then, 
\[
\norm{\tensor{A}}_p \leq \norm{\tensor{A}}_q \leq  \dimA^{\frac{1}{p}-\frac{1}{q}}\norm{\tensor{A}}_p.
\]
\end{prop}

\begin{proof}
Starting from Definition~\ref{def:norm}, we have 
\begin{align}\label{eq:pnorm2norm}
\norm{\tensor{A}}_q& =\sup\left\{ {\tensor{A}(\matr{x}_1, \dots, \matr{x}_k) \over \norm{\matr{x}_1}_q \cdots \norm{\matr{x}_k}_q}\colon \matr{x}_n\neq \matr{0}, \matr{x}_n\in\mathbb{R}^{d_n}, n\in[k] \right\} \notag \\
&=\sup\left\{ {\tensor{A}(\matr{x}_1, \dots, \matr{x}_k) \over \norm{\matr{x}_1}_p \cdots \norm{\matr{x}_k}_p} \times
{\norm{\matr{x}_1}_p \cdots \norm{\matr{x}_k}_p \over \norm{\matr{x}_1}_q \cdots \norm{\matr{x}_k}_q}\colon \matr{x}_n\neq \matr{0}, \matr{x}_n\in\mathbb{R}^{d_n}, n\in[k] \right\}.
\end{align}
For any $q\geq p\geq 1$, the equivalence of vector norms tells us 
\begin{equation} \label{eq:vectornorm}
\norm{\matr{x}}_q \leq \norm{\matr{x}}_p \leq {d}^{\frac{1}{p}-\frac{1}{q}} \norm{\matr{x}}_q, \quad\text{for all }\matr{x}\in \mathbb{R}^{d}.
\end{equation}
Applying~\eqref{eq:vectornorm} to $\matr{x}_n$, for $n \in [k]$, gives
\begin{equation}\label{eq:kvectornorm}
1 \leq \prod_{n=1}^k  {\norm{\matr{x}_n}_p \over \norm{\matr{x}_n}_q } \leq \left(\prod_{n=1}^k d_n\right)^{\frac{1}{p}-\frac{1}{q}}.
\end{equation}
Inserting~\eqref{eq:kvectornorm} into~\eqref{eq:pnorm2norm} and noting $\dimA = \prod_{n=1}^k d_n$, we find
\[
\norm{\tensor{A}}_p \leq \norm{\tensor{A}}_q \leq  \dimA^{\frac{1}{p}-\frac{1}{q}}\norm{\tensor{A}}_p,
\]
which completes the proof.
\end{proof}

\section{Partitions and general tensor unfoldings}
\label{sec:unfolding}

Any higher-order tensor can be transformed into different lower-order tensors by modifying its indices in various ways. The most common transformations are \emph{$n$-mode flattenings}, or \emph{matricizations}, which rearrange the elements of an order-$k$ tensor into a $d_n \times \prod_{i\neq n} d_i$ matrix. For example, the $n$-mode matricization of a tensor $\tensor{A}\in \mathbb{R}^{d_1\times \cdots \times d_k}$ is obtained by mapping the fixed tensor index $(i_1,\dots, i_k)$ to the matrix index $(i_n, m)$, where
\begin{equation}\label{eq:modeflattening}
m=1+\prod_{a \in [k],  a \neq n} (i_a-1)J_a, \quad \text{where\ } J_a = \prod_{l \in [a-1], l\neq n} d_l.
\end{equation}
Recently there has been much interest in studying the relationship between tensors and their matrix flattenings \cite{hu2015relations}. We present here a more general analysis by considering all possible lower-order tensor unfoldings rather than just matricizations. Using the blocks of a partition of $[k]$ to specify which modes are combined into a single mode of the new tensor, we establish a one-to-one correspondence between the set of all partitions of $[k]$ and the set of lower-order tensor unfoldings. The \emph{partition lattice} then describes the underlying relationship between possible tensor unfoldings of a tensor $\tensor{A}$.

For any $k\in\mathbb{N}_{+}$, a partition $\pi$ of $[k]$ is a collection $\{ B_1^\pi, B_2^\pi, \dots, B_\ell^\pi\}$ of disjoint, nonempty subsets (or blocks) $B_i^\pi$ satisfying $\displaystyle \cup_{i=1}^\ell B_i^\pi =[k]$. The set of all partitions of $[k]$ is denoted $\mathcal{P}_{[k]}$.  We use $|\pi|$ to denote the number of blocks in $\pi$, and $|B_i^\pi|$ to denote the number of elements in $B_i^\pi$.  We say that a partition $\pi$ is a level-$\ell$ partition if $|\pi|=\ell$.   The set of all level-$\ell$ partitions of $[k]$ is denoted by $\mathcal{P}_{[k]}^\ell$, which is a set of $S(k,\ell)$ elements, where $S(k, \ell)$ is the Stirling number of the second kind. 

The following partial order naturally relates partitions satisfying a basic compatibility constraint 
and the resulting structure plays a key role in our work.

\begin{defn}[Partition Lattice]
\label{def:lattice}
A partition $\pi_1\in\mathcal{P}_{[k]}$ is called a \emph{refinement} of $\pi_2\in\mathcal{P}_{[k]}$ if each block of $\pi_1$ is a subset of some block of $\pi_2$; conversely, $\pi_2$ is said to be a \emph{coarsening} of $\pi_1$. This relationship defines a partial order, expressed as $\pi_1\leq \pi_2$, and we say that $\pi_1$ is \emph{finer} than $\pi_2$ while $\pi_2$ is \emph{coarser} than $\pi_1$. If either $\pi_1 \leq \pi_2$ or $\pi_2 \leq \pi_1$, then $\pi_1$ and $\pi_2$ are \emph{comparable}.  
According to this partial order, the least element of $\mathcal{P}_{[k]}$ is $\zerok := \{ \{1\},\ldots,\{k\}\}$, while the greatest element is $\onek := \{\{1,\ldots,k\}\}$.
Equipped with this notion, $\mathcal{P}_{[k]}$ generates a partition lattice by connecting any two comparable partitions that differ by exactly one level.  An example is illustrated in Figure~\ref{fig:partition4}. Henceforth, $\mathcal{P}_{[k]}$ may represent either the set of all partitions of $[k]$ or the partition lattice it generates depending on context.
\end{defn}

\medskip
It is clear that many partitions are not comparable, including every pair of distinct partitions at the same level. To consider arbitrary partitions in tandem, we require an extension of this partial order. In general, for any two partitions $\pi_1, \pi_2\in\mathcal{P}_{[k]}$, we define their \emph{greatest lower bound} $\pi_1 \wedge \pi_2$ as
\[
\pi_1 \wedge \pi_2 :=\sup\{\pi\in\mathcal{P}_{[k]} \colon \pi \leq \pi_1, \pi \leq \pi_2\}.
\]
 More concretely, $\pi_1 \wedge \pi_2$ consists of the collection of all nonempty intersections of blocks in $\pi_1$ and $\pi_2$ and is unique for a given pair $(\pi_1, \pi_2)$.

\begin{figure}
\centering
\includegraphics[width=\textwidth]{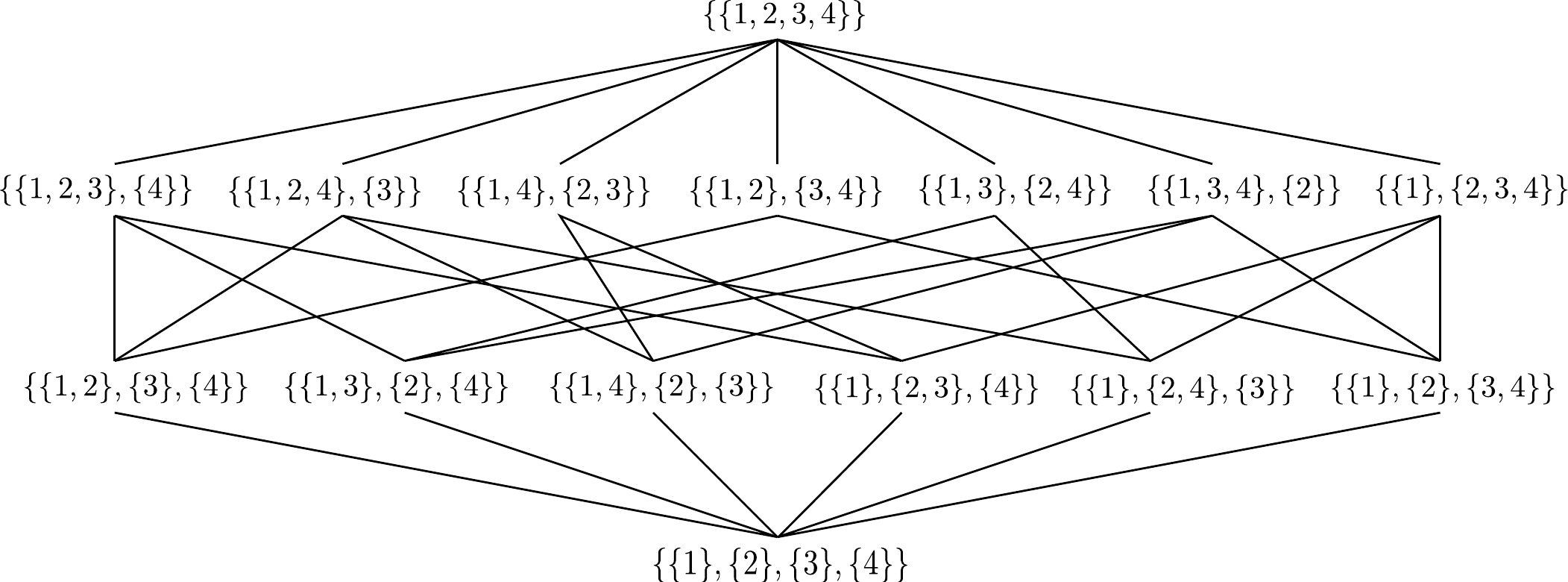}
\caption{The partition lattice $\mathcal{P}_{[4]}$.}
\label{fig:partition4}
\end{figure}

\begin{exmp}
Figure~\ref{fig:partition4} illustrates $\mathcal{P}_{[4]}$ in lattice form. Recall that an edge connects two partitions if and only if they are comparable and their levels differ by exactly one. Partitions are comparable if and only if there exists a non-reversing path between them. To clarify the components of Definition~\ref{def:lattice}, take $\pi_1=\{\{1,4\},\{2,3\}\}$ and $\pi_2=\{\{1,3,4\},\{2\}\}$. Then $\pi_1\wedge \pi_2=\{\{1,4\},\{2\},\{3\}\}$, $\bm{0}_{[4]}=\{\{1\},\{2\},\{3\},\{4\}\}$, and $\bm{1}_{[4]}=\{\{1,2,3,4\}\}$.
\end{exmp}

\medskip
The following definition generalizes the concept of $n$-mode flattenings defined in~\eqref{eq:modeflattening} to general unfoldings induced by arbitrary partitions $\pi\in\mathcal{P}_{[k]}$. 

\begin{defn}[General Tensor Unfolding] \label{def:unfold}
Let $\tensor{A}\in \mathbb{R}^{d_1\times \cdots\times d_k}$ be an order-$k$ tensor and $\pi=\{ B_1^\pi,\dots,B_\ell^\pi\}\in\mathcal{P}_{[k]}$. The partition $\pi$ defines a mapping $\phi_\pi\colon [d_1]\times \cdots \times [d_k] \to [\prod_{j\in B_1^\pi} d_j]\times \cdots \times [\prod_{j\in B_\ell^\pi} d_j]$ such that
\begin{equation}\label{def:generalunfolding}
\phi_\pi(i_1,\dots,i_k) = (m_1,\dots,m_\ell), 
\end{equation}
where
\[
m_j = 1+\prod_{r\in B_j^\pi}(i_r-1)J_r, \quad \text{where } J_r=\prod_{l\in B_j^\pi,  l < r} d_l, \quad \text{for all } j\in [\ell].
\]
Clearly, $\phi_{\pi}$ is a one-to-one mapping, so its inverse $\phi^{-1}_\pi$ is well defined. Thus $\phi_{\pi}$ induces an unfolding action $\tensor{A}\mapsto \unfold_{\pi}(\tensor{A})$ such that
\begin{equation}\label{eq:unfold}
\left(\unfold_{\pi}(\tensor{A})\right)_{(m_1,\dots, m_\ell)} = (\tensor{A})_{\phi_{\pi}^{-1}(m_1,\dots, m_\ell)},
\end{equation}
for all $(m_1,\dots, m_\ell)\in [\prod_{j\in B_1^\pi} d_j]\times \cdots \times [\prod_{j\in B_\ell^\pi} d_j],$
or, equivalently, 
\[
\left(\unfold_{\pi}(\tensor{A})\right)_{\phi_\pi (i_1,\dots, i_k)} = (\tensor{A})_{(i_1,\dots, i_k)},
\]
for all $(i_1,\dots, i_k)\in [d_1]\times \cdots \times [d_k].$ Thus $\unfold_{\pi}(\tensor{A})$ is an order-$\ell$ tensor of dimensions $(\prod_{i \in B_1^\pi } d_i, \dots,$ $\prod_{i \in B_\ell^\pi} d_i)$, and we call it the \emph{tensor unfolding of $\tensor{A}$ induced by $\pi$.} \\
\end{defn}

\begin{rmk}
By the definition of $\phi_{\pi}$ in~\eqref{def:generalunfolding}, $\unfold_{\zerok}(\tensor{A})=\tensor{A}$ and $\unfold_{\onek}(\tensor{A})=\vec(\tensor{A})$.
\end{rmk}

\begin{exmp}
Consider an order-4 tensor $\tensor{A}=\entry{a_{ijkl}}\in \mathbb{R}^{2\times 2 \times 2 \times 2 }$. We provide a subset of the possible tensor unfoldings to elucidate both the manner in which the operation works and the natural association with partitions.

\begin{itemize}
\item For $\pi = \{\{1,2\}, \{ 3 \}, \{4\}\}$, $\unfold_{\pi}(\tensor{A})$ is an order-3 tensor of dimensions $(4,2,2)$ with entries given by
\[ (\unfold_{\pi}(\tensor{A}))_{1kl} = a_{11kl}, \quad (\unfold_{\pi}(\tensor{A}))_{2kl} = a_{12kl}, \] \[ (\unfold_{\pi}(\tensor{A}))_{3kl} = a_{21kl}, \quad (\unfold_{\pi}(\tensor{A}))_{4kl} = a_{22kl}, \]
for all $(k,l) \in [2]\times[2]$.

\item For $\pi = \{ \{1,2\},\{3,4\}\}$, $\unfold_{\pi}(\tensor{A})$ is a $4\times4$ matrix
\[
\unfold_{\pi}(\tensor{A}) = \left( \begin{array}{cccc} 
a_{1111} & a_{1112} & a_{1121} & a_{1122}\\ 
a_{1211} & a_{1212} & a_{1221} & a_{1222}\\
a_{2111} & a_{2112} & a_{2121} & a_{2122} \\ 
a_{2211} & a_{2212} & a_{2221} & a_{2222}
\end{array}\right).
\]

\item For $\pi = \{ \{1,2,3\}, \{4\}\}$, $\unfold_{\pi}(\tensor{A})$ is a $8\times2$ matrix
\[
\unfold_{\pi}(\tensor{A}) = \left( 
\begin{array}{cccc} a_{1111} & a_{1112} \\ 
a_{1121} & a_{1122}\\ 
a_{1211} & a_{1212}\\ 
a_{1221} & a_{1222}\\ 
a_{2111} & a_{2112} \\ 
a_{2121} & a_{2122}\\ 
a_{2211} & a_{2212}\\ 
a_{2221} & a_{2222}
\end{array}\right). 
\]
\end{itemize}
\end{exmp}

\begin{rmk}
There are different conventions to order the elements within each transformed mode. In principle, the ordering of elements within each transformed mode is irrelevant, so we do not explicitly spell out their orderings hereafter. 
\end{rmk}

\begin{rmk} \label{rmk:Fnorm}
The unfolding operation leaves the Frobenius norm unchanged; that is, $\norm{\tensor{A}}_F=\norm{\unfold_{\pi}(\tensor{A})}_F$ for all $\pi\in\mathcal{P}_{[k]}$. More generally, the inner product remains invariant under all unfoldings: $\langle \tensor{A},\ \tensor{B}\rangle=\langle \unfold_\pi(\tensor{A}),\ \unfold_\pi(\tensor{B})\rangle$ for all $\pi\in\mathcal{P}_{[k]}$, where $\tensor{A}, \tensor{B}\in \mathbb{R}^{d_1\times \cdots\times d_k}$ are two order-$k$ tensors of the same dimensions. 
\end{rmk}

\section{Operator norm inequalities on the partition lattice}
\label{sec:main_results}

In this section, we compare the operator norms of different unfoldings of a tensor, in particular relative to that of the original tensor.  
We first focus on the spectral norm $(p=2)$ and then discuss extensions to general $l^p$-norms.

Recall that for an order-$k$ tensor $\tensor{A}\in\mathbb{R}^{d_1\times \cdots \times d_k}$, its spectral norm is defined as
\begin{equation}\label{def:spectral}
\norm{\tensor{A}}_\sigma =\sup \left\{ \tensor{A}(\matr{x}_1, \dots, \matr{x}_k)\colon \norm{\matr{x}_n}_2 = 1, \matr{x}_n\in\mathbb{R}^{d_n}, n\in [k] \right\}.
\end{equation}
The maximization of the polynomial form $\tensor{A}(\matr{x}_1, \dots, \matr{x}_k)$ on the unit sphere is closely related to the best, in the least-square sense, rank-1 tensor approximation. Specifically, for an order-$k$ tensor $\tensor{A}\in\mathbb{R}^{d_1\times \cdots \times d_k}$, the problem of determining the spectral norm is equivalent to finding a scalar $\lambda$ and a rank-1 norm-1 tensor $\matr{x}_1\otimes \cdots \otimes \matr{x}_k$ that minimize the function
\[
f(\matr{x}_1, \ldots,\matr{x}_k)=\norm{\tensor{A}-\lambda \matr{x}_1\otimes\cdots\otimes \matr{x}_k}_F.
\]
The corresponding value of $\lambda$ is equal to $\norm{\tensor{A}}_\sigma$. 
Since $\matr{x}_1\otimes\cdots\otimes \matr{x}_k$ must have the same order and dimensions as $\tensor{A}$, the rank-1 condition becomes less strict the more unfolded the tensor. In particular, for the vectorized unfolding, $\vec(\tensor{A})$, the best rank-1 tensor approximation is simply $\vec(\tensor{A})$ itself. For unfoldings into a matrix, $\text{Mat}(\tensor{A})$, the best rank-1 approximation is the outer product of the leading left and the right singular vectors of $\text{Mat}(\tensor{A})$. For higher-order unfoldings, the closed form of the best rank-1 approximation is not known in general. Nevertheless, the set of the rank-1 tensors over which the supremum is taken in \eqref{def:spectral} becomes more restricted. This observation implies that the spectral norm of a tensor unfolding decreases as the order of the unfolded tensor increases; that is, the spectral norm preserves the partial order on partitions.

\begin{prop}[Monotonicity]  \label{monotonicity} 
For all partitions $\pi_1, \pi_2\in\mathcal{P}_{[k]}$ satisfying $\pi_1 \leq \pi_2$, 
\begin{equation}
\norm{ {\unfold_{\pi_1}(\tensor{A})}}_\sigma \leq \norm{ {\unfold_{\pi_2}(\tensor{A})}}_\sigma.
\end{equation}
In particular, we have the global extrema
\begin{align}
\norm{\tensor{A}}_\sigma &=\norm{ {\unfold_{\zerok}(\tensor{A})}}_\sigma  =\min_{\pi\in \mathcal{P}_{[k]}} \norm{ \unfold_\pi(\tensor{A})}_\sigma,\\
\quad \norm{\tensor{A}}_F &=\norm{ {\unfold_{\onek}(\tensor{A})}}_\sigma=\max_{\pi\in \mathcal{P}_{[k]}} \norm{ \unfold_\pi(\tensor{A})}_\sigma.
\end{align}
\end{prop}

\begin{proof}

Suppose $\pi_1, \pi_2\in\mathcal{P}_{[k]}$, and $\pi_1$ is a one-step refinement of $\pi_2$. Without loss of generality, assume $\pi_2$ is obtained by merging two blocks $B_1, B_2\in \pi_1$ into a single block $B\in \pi_2$. Let $\unfold_{\pi_1}(\tensor{A})\in\mathbb{R}^{d_1\times \cdots \times d_\ell}$ denote the tensor unfolding induced by $\pi_1=\{B_1,\ldots,B_\ell\}$. Then, $\unfold_{\pi_2}(\tensor{A})$ is a $(d_1\cdot d_2, d_3,\ldots, d_\ell)$-dimensional tensor. By definition of the spectral norm,
\[
\begin{aligned}
\norm{\unfold_{\pi_1}(\tensor{A})}_\sigma&=\sup \big \{\left \langle\unfold_{\pi_1}(\tensor{A}), \  \matr{x}_{1}\otimes\cdots \otimes \matr{x}_{\ell}\right \rangle \colon \matr{x}_{n}\in S^{d_n-1}, n \in [\ell] \} \\
&= \sup \big \{\left  \langle \unfold_{\pi_2}(\tensor{A}), \  \vec(\matr{x}_{1}\otimes \matr{x}_2)\otimes \matr{x}_3\otimes \cdots \otimes \matr{x}_{\ell} \right \rangle \colon \matr{x}_{n}\in S^{d_n-1}, n \in [\ell] \}\\
&\leq \sup \big \{\left \langle\unfold_{\pi_2}(\tensor{A}), \  \matr{y}\otimes\matr{x}_3\cdots \otimes \matr{x}_\ell \right \rangle \colon \matr{y} \in S^{d_1\cdot d_2-1},  \matr{x}_{n}\in S^{d_n-1}, n=3,\ldots, \ell \} \\
&= \norm{\unfold_{\pi_2}(\tensor{A})}_\sigma,
\end{aligned}
\]
where the third line comes from the fact that the set $\{ \vec(\matr{x}_1 \otimes \matr{x}_1)\colon (\matr{x}_1,\matr{x}_2) \in S^{d_1-1}\times S^{d_2-1}\}$ is contained in the set $\left\{\matr{y}\colon \matr{y}\in S^{d_1\cdot d_2-1}\right\}$.

In general, if $\pi_1 \leq \pi_2$, we can obtain ${\unfold_{\pi_2}(\tensor{A})}$ from ${\unfold_{\pi_1}(\tensor{A})}$ by a series of single unfoldings. Hence, applying the above arguments to these successive unfoldings gives the desired result.
\end{proof}

The following lemma will play a key role in proving our main results:

\begin{lem}[One-Step Inequality]
\label{onestep}
For $1 < \ell \leq k$, let $\tensor{B}\in \mathbb{R}^{d_1\times \cdots\times d_\ell}$ be an order-$\ell$ tensor unfolding of an order-$k$ tensor $\tensor{A}$ induced by the partition $\{B_1, \dots, B_\ell\}\in\mathcal{P}_{[k]}$, and let $\tensor{C}$ be an order-$(\ell-1)$ tensor unfolding of $\tensor{B}$ induced by merging blocks $B_i$ and $B_j$ for some $i, j\in [\ell]$. Then,
\begin{equation}\label{lem:inequality}
\min \left(d_i, d_j \right)^{-1/2}\norm{\tensor{C}}_\sigma \leq \norm{\tensor{B}}_\sigma\leq \norm{\tensor{C}}_\sigma.
\end{equation}
\end{lem}

\begin{proof}
The upper bound follows readily from Proposition~\ref{monotonicity}. To prove the lower bound, without loss of generality, assume $\tensor{C}$ corresponds to the merging of blocks $B_{\ell-1}$ and $B_\ell$ so that $\tensor{C}$ is a $(d_1, \dots, d_{\ell-2}, d_{\ell-1} \cdot d_\ell)$-dimensional tensor. Note that $S^{d_1-1}\times\dots \times S^{d_{\ell-2}-1}\times S^{d_{\ell-1}d_\ell-1} $ is a compact set, so the supremum~\eqref{def:spectral} is attained in that set for $\tensor{C}$. Then, there exists $(\matr{x}_1^*, \dots, \matr{x}_{\ell-2}^*,\ \matr{y}^*) \in S^{d_1-1}\times\dots \times S^{d_{\ell-2}-1} \times S^{d_{\ell-1}d_\ell-1} $ such that
\begin{equation}\label{eq:maximizer}
\norm{\tensor{C}}_{\sigma}
=\langle \tensor{C}, \ \matr{x}^*_1 \otimes \cdots \otimes \matr{x}^*_{\ell-2}\otimes \matr{y}^*\rangle.
\end{equation}
Define $\tensor{C}^*=\tensor{C}(\matr{x}^{*}_{1},\dots, \matr{x}^{*}_{\ell-2}, \matr{I}_{d_{\ell-1}d_\ell}) \in \mathbb{R}^{d_{\ell-1}d_{\ell}}$. By the self duality of the Frobenius norm in $\mathbb{R}^{d_{\ell-1}d_\ell}$, we can rewrite~\eqref{eq:maximizer} as
\begin{align} \label{eq:partition1}
\norm{\tensor{C}}_\sigma&=\sup \left\{ \langle \tensor{C}^*,\ \matr{y}\rangle\colon \matr{y}\in S^{d_{\ell-1}d_\ell-1} \right\} =\norm{\tensor{C}^*}_F.
\end{align}
On the other hand, if we define $\text{Mat}(\tensor{C}^*) :=\tensor{B}(\matr{x}_1^*, \dots, \matr{x}_{\ell-2}^*, \matr{I}_{\ell-1},\matr{I}_\ell)\in\mathbb{R}^{d_{\ell-1}\times d_\ell}$, then $\tensor{C}^*$ is simply the vectorization of $\text{Mat}(\tensor{C}^*)$.  Hence, by Remark~\ref{rmk:Fnorm}, we obtain
\begin{equation}\label{eq:matrixization}
\norm{\tensor{C}^*}_F=\norm{\text{Mat}(\tensor{C}^*)}_F.
\end{equation} 
Using the definition of $\text{Mat}(\tensor{C}^*)$, we can write
\begin{align}\label{lemma1ineq}
\norm{\tensor{B}}_\sigma & \geq  \sup \left\{ \langle \tensor{B},\ \matr{x}^*_1\otimes\matr{x}^*_2 \otimes \cdots \otimes \matr{x}^*_{\ell-2}\otimes \matr{x}_{\ell-1} \otimes \matr{x}_\ell \rangle \colon  (\matr{x}_{\ell-1},\matr{x}_\ell) \in S^{d_{\ell-1}-1}\times S^{d_{\ell}-1}  \right\}  \notag \\
&=  \sup \left\{ \left \langle \text{Mat}(\tensor{C}^*),\ \matr{x}_{\ell-1} \otimes \matr{x}_\ell  \right \rangle \colon (\matr{x}_{\ell-1},\matr{x}_\ell) \in S^{d_{\ell-1}-1}\times S^{d_{\ell}-1}  \right\}\notag \\
&=\norm{\text{Mat}(\tensor{C}^*)}_\sigma.
\end{align}
Recall that $\text{Mat}(\tensor{C}^*)$ is a matrix of size $d_{\ell-1}\times d_\ell$, meaning~\cite{horn2012matrix}
 \begin{equation}\label{eq:matrixnorm}
\norm{\text{Mat}({\tensor{C}^*})}_\sigma \geq  \min\left (d_{\ell-1} , d_{\ell} \right )^{-1/2}\norm{\text{Mat}({\tensor{C}^*})}_F.
\end{equation}
 Using~\eqref{lemma1ineq} in conjunction with~\eqref{eq:matrixnorm},~\eqref{eq:matrixization}, and~\eqref{eq:partition1}, we then obtain
 \begin{align}
\norm{\tensor{B}}_\sigma & \geq  \min \left(d_{\ell-1}, d_{\ell} \right)^{-1/2}\norm{\text{Mat}(\tensor{C}^*)}_F \\
&=\min\left(d_{\ell-1}, d_\ell \right)^{-1/2}\norm{\tensor{C}^*}_F\\
& =  \min \left(d_{\ell-1}, d_{\ell} \right)^{-1/2}\norm{\tensor{C}}_\sigma,
\end{align}
which proves the lower bound.
\end{proof}

\begin{rmk}
Both bounds in the one-step inequality \eqref{lem:inequality} are sharp. The sharpness of the upper bound will be discussed in Section~\ref{sec:OD}. For the lower bound, consider an order-$\ell$ tensor unfolding $\tensor{B}= \left(\sum_{i=1}^{d_1}\matr{e}_{1,i}\otimes \matr{e}_{2,i} \right)\otimes \tensor{D}\in\mathbb{R}^{d_1\times d_2\times \cdots \times d_\ell}$, 
where $\{\matr{e}_{1,i}: i\in[d_1]\}$ is the standard orthonormal basis of $\mathbb{R}^{d_1}$, $\{\matr{e}_{2,i}: i\in[d_1]\}$ is a set of $d_1$ standard basis vectors in $\mathbb{R}^{d_2}$, and $\tensor{D}$ is an arbitrary $(d_3, \ldots, d_\ell)$-dimensional tensor.  Assume $d_1\leq d_2$ and consider the two blocks $B_1=\{1\}, B_2=\{2\}$. By the unfolding operation specified in Lemma~\ref{onestep} with $B_1$ and $B_2$, $\tensor{C}=\left(\sum_{i=1}^{d_1}\vec(\matr{e}_{1,i}\otimes\matr{e}_{2,i})\right)\otimes \tensor{D}\in\mathbb{R}^{d_1\cdot d_2\times d_3 \times \cdots\times d_\ell}$, which is an order-$(\ell-1)$ tensor unfolding by merging the first two modes of $\tensor{B}$ into a single mode. It follows that $\norm{\tensor{B}}_\sigma=\norm{\tensor{D}}_\sigma$ and $\norm{\tensor{C}}_\sigma=\sqrt{d_1}\norm{\tensor{D}}_\sigma$, and therefore the left-hand-side inequality in \eqref{lem:inequality} is saturated.
\end{rmk}

\medskip
More generally, we can establish inequalities relating the spectral norms of two tensor unfoldings corresponding to two arbitrary partitions $\pi_1, \pi_2\in\mathcal{P}_{[k]}$ that are not necessarily comparable. To do so, we must first introduce the following definition. 

\begin{defn} \label{def:vol}
Given an order-$k$ tensor $\tensor{A}\in\mathbb{R}^{d_1\times \cdots\times d_k}$, we define the map $\dim_{\tensor{A}}\colon \mathcal{P}_{[k]} \times \mathcal{P}_{[k]} \rightarrow \mathbb{N}_{+}$ as
\begin{equation}
\dim_{\tensor{A}}(\pi_1, \pi_2)  =  \prod\limits_{B\in \pi_1} \max\limits_{B'\in\pi_2} D_\tensor{A}(B, B' ), \quad \text{where} \quad D_\tensor{A}(B,B')=\prod\limits_{n\in B \cap B' }d_{n},
\end{equation}
where $\pi_1,\pi_2 \in \mathcal{P}_{[k]}$.

\medskip
We label this quantity as $\dim_{\tensor{A}}(\cdot, \cdot)$ because it involves a product of a subset of the dimensions of $\tensor{A}$, and $\dim_{\tensor{A}}(\pi_1, \pi_2) \leq \dimA$ with equality only when $\pi_1 = \pi_1 \wedge \pi_2$. Intuitively, $\dim_{\tensor{A}}(\pi_1, \pi_2)$ reflects the overlap between the unfoldings induced by $\pi_1$ and $\pi_1 \wedge \pi_2$. Example~\ref{exmp:Vol} presents a concrete illustration. 
\end{defn}

\begin{rmk}
We set $D_\tensor{A}(B, B')=0$ when $B \cap B' =\emptyset$. This does not affect the multiplication in Definition~\ref{def:vol} because a block of $\pi_1$ cannot be disjoint from every block of $\pi_2$. Note that $D_\tensor{A}(B, B')=D_\tensor{A}(B',B)$, but in general $\dim_{\tensor{A}}(\pi_1, \pi_2)\neq \dim_{\tensor{A}}(\pi_2, \pi_1)$.
\end{rmk}

\begin{exmp}
\label{exmp:Vol}
To illustrate the above map, let $\tensor{A}\in \mathbb{R}^{d_1 \times d_2\times d_3\times d_4}$ and  consider the partitions $\pi_1 = \{\{1,2\},\{3,4\}\}$ and $\pi_2 = \{\{1,2,3\},\{4\}\}$, for which $\pi_1 \wedge \pi_2 = \{\{1,2\}, \{3\}, \{4\} \}.$ From Definition~\ref{def:vol},
\begin{align}
\begin{array}{lllll}
D_\tensor{A}(\{1,2\},\{1,2,3\}) & = d_1d_2 \\
D_\tensor{A}(\{3,4\},\{1,2,3\}) &= d_3 \\
D_\tensor{A}( \{1,2\},\{4\}) &= 0 \\
D_\tensor{A}(\{3,4\},\{4\}) &= d_4.
\end{array}
\end{align}
Then,
\begin{align}
 \dim_{\tensor{A}} (\pi_1, \pi_2) &=\max \big\{D_\tensor{A}(\{1,2\},\{1,2,3\}) ,  D_\tensor{A}(\{1,2\},\{4\})\big\} \\
 & \hspace{1cm}\times \max \big\{D_\tensor{A}(\{3,4\},\{1,2,3\})  ,  D_\tensor{A}(\{3,4\},\{4\})\big\}  \\
 &=d_1d_2 \max\{d_3, d_4\}.
\end{align}
Exchanging arguments, we find
\begin{align}
 \dim_{\tensor{A}} (\pi_2, \pi_1) &=  \max \big\{D_\tensor{A}(\{1,2,3\},\{1,2\})  ,  D_\tensor{A}(\{1,2,3\},\{3,4\})\big\}\\
 & \hspace{1cm} \times  \max \big\{D_\tensor{A}(\{4\},\{1,2\})  ,  D_\tensor{A}(\{4\},\{3,4\})\big\}  \\
 &=d_4\max\{d_1 d_2, d_3\}.
\end{align}
\end{exmp}

\begin{rmk}
As in Lemma~\ref{onestep}, if $\pi_2$ is a one-step coarsening of $\pi_1$ obtained by merging two blocks $B^{\pi_1}_{i}$ and $B^{\pi_1}_{j}$ of dimensions $d_{i}$ and $d_{j}$ into a single block $B^{\pi_2}_1$, then
\begin{align}
D_\tensor{A}(B^{\pi_1}_i,B^{\pi_2}_1) &= d_{i}, \\
D_\tensor{A}(B^{\pi_1}_j,B^{\pi_2}_1) &= d_{j}.
\end{align}
Hence, Lemma~\ref{onestep} can be written as
\begin{equation}
\min \big\{D_\tensor{A}(B^{\pi_1}_i, B^{\pi_2}_1), D_\tensor{A}(B^{\pi_1}_j, B^{\pi_2}_1)\big\}^{-1/2}\norm{\unfold_{\pi_2}(\tensor{A})}_\sigma \leq \norm{\unfold_{\pi_1}(\tensor{A})}_\sigma\leq \norm{\unfold_{\pi_2}(\tensor{A})}_\sigma.
\end{equation}
\end{rmk}

\medskip
Having introduced Definition~\ref{def:vol}, we can now state our main result on how the spectral norms of two arbitrary unfoldings of a tensor are related:

\begin{thm}[Spectral norm inequalities] \label{thm:main}
Let $\tensor{A}\in\mathbb{R}^{d_1\times \cdots \times d_k}$ be an arbitrary order-$k$ tensor, and $\pi_1, \pi_2$ any two partitions in $\mathcal{P}_{[k]}$. Then, 
$$
\left[\dimA \over \dim_{\tensor{A}} (\pi_1, \pi_2)\right]^{-1/2} \norm{\unfold_{\pi_1}({\tensor{A})}}_\sigma \leq \norm{\unfold_{\pi_2}({\tensor{A})}}_\sigma \leq \left[ {\dimA \over \dim_{\tensor{A}} (\pi_2, \pi_1)} \right]^{1/2} \norm{\unfold_{\pi_1}({\tensor{A})}}_\sigma.
$$
\end{thm}

\medskip
\begin{rmk}\ 
\begin{enumerate}
\item[(a)] Note that $\pi_1$ and $\pi_2$ need not be comparable.
\item[(b)]  If  $d_n = d$ for all $n \in [k]$, then the result reduces to
\[
d^{-c_1/2 } \norm{\unfold_{\pi_1}({\tensor{A})}}_\sigma \leq \norm{\unfold_{\pi_2}({\tensor{A})}}_\sigma \leq d^{c_2/2 } \norm{\unfold_{\pi_1}({\tensor{A})}}_\sigma, 
\]
where $c_1 = (k - \sum_{B \in \pi_1} \max_{B' \in \pi_2} | B \cap B'|)$ and $c_2 = (k - \sum_{B \in \pi_2} \max_{B' \in \pi_1} | B \cap B'|)$.
\item[(c)]  For $k=4$, the above inequalities in (b) are sharp.  For example, 
consider the tensor $\tensor{A}=\matr{I}_d\otimes \matr{I}_d$, and partitions $\pi=\{\{1, 2\}, \{3,4\}\}$ and $\pi'=\{\{1\},\{2\},\{3\},\{4\}\}$, for which we have $\norm{\unfold_{\pi}(\tensor{A})}_\sigma=d$ and $\norm{\unfold_{\pi'}({\tensor{A})}}_\sigma = 1$.
If $\pi_1=\pi$ and $\pi_2=\pi'$, then $c_1=2$ and $d^{-c_1/2} \norm{\unfold_{\pi_1}(\tensor{A})}_\sigma = \norm{\unfold_{\pi_2}({\tensor{A})}}_\sigma$.  On the other hand, if $\pi_1=\pi'$  and $\pi_2=\pi$, then we have $c_2=2$ and $\norm{\unfold_{\pi_2}({\tensor{A})}}_\sigma = d^{c_2/2} \norm{\unfold_{\pi_1}({\tensor{A})}}_\sigma$.  This particular tensor is further discussed later in Example~\ref{exam:bound}. 

\end{enumerate}
\end{rmk}

\medskip
\begin{proof}[Proof of Theorem~\ref{thm:main}]
The main idea is to apply Lemma~\ref{onestep} to some appropriate sequence of partitions connecting $\pi_1$ and $\pi_2$ in the partition lattice. To do so, we consider $\unfold_{\pi_1 \wedge \pi_2}({\tensor{A})}$ and compare its spectral norm to that of  $\unfold_{\pi_1}({\tensor{A})}$.  Since $\pi_1\wedge \pi_2 \leq \pi_1$, from Proposition~\ref{monotonicity} we have
\begin{equation}\label{ineqproofp1}
\norm{\unfold_{\pi_1\wedge \pi_2}({\tensor{A})}}_\sigma   \leq \norm{\unfold_{\pi_1}({\tensor{A})}}_\sigma.
\end{equation}
Let $\pi_1=\{B^{\pi_1}_i\colon i\in[\ell_1]\}$ and $\pi_2=\{B^{\pi_2}_j\colon j\in[\ell_2]\}$, and note that $\pi_1\wedge \pi_2=\{B^{\pi_1}_i\cap B^{\pi_2}_j \colon i\in[\ell_1], j\in[\ell_2]\}$.  Now, order the blocks in $\pi_2$ such that
\begin{equation}\label{mergelow}
D_\tensor{A}(B^{\pi_1}_1,B^{\pi_2}_{m_1}) \geq D_\tensor{A}(B^{\pi_1}_1,B^{\pi_2}_{m_2}) \geq \cdots \geq  D_\tensor{A}(B^{\pi_1}_1,B^{\pi_2}_{m_{\ell_2}}),
\end{equation}
and define $B_{1j} = B^{\pi_1}_1\cap B^{\pi_2}_{m_j}$.
Consider a sequence $(\tensor{T}_1,\ \tensor{T}_2,\ \dots,\ \tensor{T}_{\ell_2-1})$ of unfoldings of $\tensor{A}$ where $\tensor{T}_j$, for $1\leq j \leq \ell_2-1$, is obtained from the tensor $\unfold_{\pi_1\wedge \pi_2}({\tensor{A})}$ by an unfolding operation corresponding to merging the blocks $B_{1,1},\ldots,B_{1,j+1}$ into a single block.
Using Lemma~\ref{onestep} and~\eqref{mergelow}, we obtain
\begin{align}
\norm{\unfold_{\pi_1\wedge \pi_2}({\tensor{A})}}_\sigma  & \geq  \min\{ D_\tensor{A}(B^{\pi_1}_1, B^{\pi_2}_{m_1}), D_\tensor{A}(B^{\pi_1}_1, B^{\pi_2}_{m_2})\}^{-1/2}\norm{\tensor{T}_1}_\sigma \\ 
& =\left[D_\tensor{A}(B^{\pi_1}_1, B^{\pi_2}_{m_2})\right]^{-1/2}\norm{\tensor{T}_1}_\sigma.
\end{align}
Similarly for all $i\in [\ell_2 -2]$,
\begin{align}
\norm{\tensor{T}_i}_\sigma & \geq \min \Big\{D_\tensor{A}(B^{\pi_1}_1, B^{\pi_2}_{m_{i+2}}) \ , \  \prod\limits_{j=1}^{i+1}D_\tensor{A}(B^{\pi_1}_1, B^{\pi_2}_{m_j}) \Big\}^{-1/2}\norm{\tensor{T}_{i+1}}_\sigma   \\ 
& =  \big[D_\tensor{A}(B^{\pi_1}_1, B^{\pi_2}_{m_{i+2}})\big]^{-1/2}\norm{\tensor{T}_{i+1}}_\sigma.
\end{align}
Combining these inequalities gives
\begin{align}
\norm{\unfold_{\pi_1\wedge \pi_2}({\tensor{A})}}_\sigma   &\geq 
\left[\prod\limits_{j=2}^{\ell_2} D_\tensor{A}(B^{\pi_1}_1, B^{\pi_2}_{m_j})\right]^{-1/2} \norm{\tensor{T}_{\ell_2 -1}}_\sigma \\
& =  \left[ \frac{\prod\limits_{j=1}^{\ell_2} D_\tensor{A}(B^{\pi_1}_1, B^{\pi_2}_j)}{\max\limits_{j \in [\ell_2]} \{D_\tensor{A}(B^{\pi_1}_1, B^{\pi_2}_j)\} } \right]^{-1/2}\norm{\tensor{T}_{\ell_2 -1}}_\sigma.
\end{align}
We can iterate the same line of argument with $\{ B^{\pi_1}_i\cap B^{\pi_2}_{m_j} \colon  j\in [\ell_2]\}$ for $i = 2,\ldots,\ell_1$ to obtain
\begin{align}
\norm{\unfold_{\pi_1\wedge \pi_2}({\tensor{A})}}_\sigma   &\geq\prod\limits_{i=1}^{\ell_1}\left[ \frac{\prod\limits_{j=1}^{\ell_2} D_\tensor{A}(B^{\pi_1}_i, B^{\pi_2}_j)}{\max\limits_{j \in [\ell_2]} \{D_\tensor{A}(B^{\pi_1}_i, B^{\pi_2}_j)\} } \right]^{-1/2}\norm{\unfold_{\pi_1}({\tensor{A})}}_\sigma \\
&=\left[\frac{ \prod\limits_{i=1}^{\ell_1}\prod\limits_{j=1}^{\ell_2}D_\tensor{A}(B^{\pi_1}_i, B^{\pi_2}_j)}{\prod\limits_{i=1}^{\ell_1}\max\limits_{j \in [\ell_2]} \{D_\tensor{A}(B^{\pi_1}_i, B^{\pi_2}_j)\}} \right]^{-1/2}\norm{\unfold_{\pi_1}({\tensor{A})}}_\sigma \\
&=  \left[{\prod_{n\in[k]} d_n \over \dim_{\tensor{A}}(\pi_1, \pi_2)}\right]^{-1/2} \norm{\unfold_{\pi_1}({\tensor{A})}}_\sigma.
\end{align}
Together with~\eqref{ineqproofp1}, this last inequality means
\begin{equation}
\label{eq:triple1}
\norm{\unfold_{\pi_1\wedge \pi_2}({\tensor{A})}}_\sigma   \leq \norm{\unfold_{\pi_1}({\tensor{A})}}_\sigma  \leq \left[ {\prod_{n\in[k]} d_n \over \dim_{\tensor{A}}(\pi_1, \pi_2)}\right]^{1/2} \norm{\unfold_{\pi_1\wedge \pi_2}({\tensor{A})}}_\sigma.
\end{equation}
By symmetry, 
\begin{equation}\label{eq:triple2}
\norm{\unfold_{\pi_1\wedge \pi_2}({\tensor{A})}}_\sigma   \leq \norm{\unfold_{\pi_2}({\tensor{A})}}_\sigma \leq  \left[ {\prod_{n\in[k]} d_n \over \dim_{\tensor{A}}(\pi_2, \pi_1)}\right]^{1/2} \norm{\unfold_{\pi_1\wedge \pi_2}({\tensor{A})}}_\sigma.
\end{equation}
Finally, combining~\eqref{eq:triple1} and~\eqref{eq:triple2} completes the proof. 
\end{proof}

\medskip

We may immediately establish several corollaries of Theorem~\ref{thm:main}. 

\begin{cor}\label{cor:frobenius}
All order-$k$ tensors $\tensor{A}\in\mathbb{R}^{d_1\times \cdots \times d_k}$ satisfy
\[
\norm{\tensor{A}}_F\leq \left[\dimA \over\max_{n\in[k]} {d_n}\right]^{1/2}\norm{\tensor{A}}_\sigma.
\]
\end{cor}
\begin{proof}
Taking $\pi_1=\zerok$ and $\pi_2=\onek$ in Theorem~\ref{thm:main} yields the result. 
\end{proof}

Corollary~\ref{cor:frobenius} gives the worst-case ratio of the Frobenius norm to the spectral norm for an arbitrary tensor.  This ratio is sharper than the bound recently found by Friedland and Lim~\cite[Lemma~5.1]{friedland2016computational},  namely $\norm{\tensor{A}}_F\leq \dimA^{1/2}\norm{\tensor{A}}_\sigma$.

We now give a set of inequalities comparing the spectral norms of unfoldings at level $\ell$ to that at either level $k$ or level $1$. For ease of exposition, we assume $d_n=d$ for all $n\in[k]$.

\begin{cor}[Bottom-Up Inequality]  \label{thm:bottomup}
 Let $\tensor{A}\in\mathbb{R}^{d\times \cdots \times d}$ be an order-$k$ tensor with the same dimension $d$ in all modes. For all levels $1\leq \ell \leq k$ and partitions $\pi \in \mathcal{P}_{[k]}^\ell$,
\begin{equation}\label{eq:bottomup}
 d^{-(k-\ell)/2}  \max_{\pi\in \mathcal{P}^{\ell}_{[k]}} \norm{ {\unfold_{\pi}(\tensor{A})}}_\sigma \leq \norm{\tensor{A}}_\sigma \leq \min_{\pi\in\mathcal{P}^{\ell}_{[k]}} \norm{ \unfold_\pi(\tensor{A})}_\sigma.
\end{equation}
\end{cor}
\begin{proof} Take $\pi_2 = \zerok$ in Theorem~\ref{thm:main}. \end{proof}

\begin{rmk}
The existing work that is most closely related to our own is that of Hu's \cite{hu2015relations}, in which the author bounds the nuclear norm of a tensor by that of its matricization. Since the nuclear norm and spectral norm are dual to each other in tensor space, many of our results apply to the nuclear norm as well. Particularly, letting $\ell=2$ in the bottom-up inequality \eqref{eq:bottomup} reproduces Hu's results. 
\end{rmk}

\begin{cor}[Top-Down Inequality]  \label{thm:topdown}
Let $\tensor{A}\in\mathbb{R}^{d\times \cdots \times d}$ be an order-$k$ tensor  with the same dimension $d$ in all modes. For all levels $1\leq \ell \leq k$ and partitions $\pi \in \mathcal{P}_{[k]}^\ell$,
\begin{equation} \label{eq:topdown}
d^{- (k- \max_{i \in [\ell]}|B_i^{\pi}|)/2}\norm{\tensor{A}}_F \leq \norm{\unfold_{\pi}({\tensor{A})}}_\sigma \leq \norm{\tensor{A}}_F.
\end{equation}
\end{cor}
\begin{proof} Take $\pi_1 = \onek$ in Theorem~\ref{thm:main}. \end{proof}

\begin{cor}
Let $\tensor{A}\in\mathbb{R}^{d\times \cdots \times d}$ be an order-$k$ tensor  with the same dimension $d$ in all modes.  For all levels $1\leq \ell \leq k$,
\begin{equation} \label{eq:topdown2}
d^{-(k-\lceil {k/ \ell} \rceil)/2}\norm{\tensor{A}}_F \leq \min_{\pi\in \mathcal{P}^{\ell}_{[k]}}\norm{\unfold_{\pi}({\tensor{A})}}_\sigma.
\end{equation}
\end{cor} 
\begin{proof} Note that the minimum of the maximal block sizes across all level-$\ell$ partitions of $[k]$ is $\displaystyle \min_{\pi \in \mathcal{P}_{[k]}^\ell} \max_{B \in \pi } |B| = \lceil{ k / \ell}\rceil$, and apply Corollary~\ref{thm:topdown}.
\end{proof}

\medskip
The above corollaries bound the amount by which norms can vary over a specific level $\ell$. They imply that the ratios $\norm{\unfold_{\pi}({\tensor{A})}}_\sigma / \norm{\tensor{A}}_\sigma$ and $\norm{\unfold_{\pi}({\tensor{A})}}_\sigma / \norm{\tensor{A}}_F$ fall in the intervals $[1,\ d^{(k-\ell)/2}]$ and $[d^{-(k-\lceil {k/\ell}\rceil)/2},\ 1]$, respectively. Therefore, in the worst case, $\norm{\unfold_{\pi}({\tensor{A})}}_\sigma$ only recovers $\norm{\tensor{A}}_\sigma$ or $\norm{\tensor{A}}_F$ at $\poly(d)$ precision.  
Note that the factor $d^{(k-\ell)/2}$ has an exponent linear in the difference between the orders of the original tensor and its flattening. This means that the potential deviation between their spectral norms depends only on the difference in their orders rather than the actual orders themselves, and that the deviation accumulates in multiplicative fashion with a loss of $\sqrt{d}$ in precision at each level.
In contrast, the factor $d^{-(k-\lceil {k/ \ell}\rceil)/2}$ depends on more than just the gap between $k$ and $\ell$, with a larger impact for unfoldings with orders close to $k$. 

We provide a low-order example that reaches the $\poly(d)$ scaling factor in Corollary~\ref{thm:bottomup}.

\begin{exmp}\label{exam:bound}
Consider the order-4 tensor $\tensor{A}=\matr{I}_d\otimes \matr{I}_d$. Straightforward calculation shows that $\norm{\tensor{A}}_\sigma=1$.
Furthermore, by symmetry, the spectral norm of an unfolding induced by any partition in $\mathcal{P}^2_{[4]}$ or $\mathcal{P}^3_{[4]}$ must fall into one of the following five representative cases:
\begin{itemize}
\item For $\pi_1=\{\{1\},\{2\},\{3,4\}\}\in\mathcal{P}^3_{[4]}$, $\norm{\unfold_{\pi_1}(\tensor{A})}_\sigma=\sqrt{d}$. 
\item For $\pi_2=\{\{1\},\{3\},\{2,4\}\}\in\mathcal{P}^3_{[4]}$, $\norm{\unfold_{\pi_2}(\tensor{A})}_\sigma=1$.
\item For $\pi_3=\{\{1,2\},\{3,4\}\}\in\mathcal{P}^2_{[4]}$, $\norm{\unfold_{\pi_3}(\tensor{A})}_\sigma=d$. 
\item For $\pi_4=\{\{1\},\{2,3,4\}\}\in\mathcal{P}^2_{[4]}$, $\norm{\unfold_{\pi_4}(\tensor{A})}_\sigma=\sqrt{d}$.
\item For $\pi_5=\{\{1,3\},\{2,4\}\}\in\mathcal{P}^2_{[4]}$, $\norm{\unfold_{\pi_5}(\tensor{A})}_\sigma=1$. 
\end{itemize}
Therefore, 
$$
\max_{\pi\in\mathcal{P}_{[4]}^3} \norm{\unfold_\pi (\tensor{A})}_\sigma=d^{(4-3)/2}\norm{\tensor{A}}_\sigma \quad \text{and}\quad \min_{\pi\in\mathcal{P}_{[4]}^3} \norm{\unfold_\pi (\tensor{A})}_\sigma=\norm{\tensor{A}}_\sigma.
$$ 
$$
\max_{\pi\in\mathcal{P}_{[4]}^2} \norm{\unfold_\pi (\tensor{A})}_\sigma=d^{{(4-2)}/2}\norm{\tensor{A}}_\sigma \quad \text{and}\quad \min_{\pi\in\mathcal{P}_{[4]}^2} \norm{\unfold_\pi (\tensor{A})}_\sigma=\norm{\tensor{A}}_\sigma.
$$
\end{exmp}

\bigskip
We conclude this section by generalizing Theorem~\ref{thm:main} to $l^p$-norms. 

\begin{thm}[$l^p$-norm inequalities]\label{thm:main2}
Let $\tensor{A}\in \mathbb{R}^{d_1\times \cdots\times d_k}$ be an arbitrary order-$k$ tensor, and $\pi_1, \pi_2$ any two partitions in $\mathcal{P}_{[k]}$.  Then,
\begin{enumerate}
\item[(a)] For any $1\leq p\leq 2$,
\[
{[\dimA]^{-1/p} \over \left[\dim_{\tensor{A}} (\pi_1, \pi_2)\right]^{-1/2}} \norm{\unfold_{\pi_1}({\tensor{A})}}_p \leq \norm{\unfold_{\pi_2}({\tensor{A})}}_p \leq  {[\dimA]^{1/p} \over \left[\dim_{\tensor{A}} (\pi_2, \pi_1) \right]^{1/2}} \norm{\unfold_{\pi_1}({\tensor{A})}}_p.
\]
\item[(b)] For any $2\leq p\leq \infty$,
\[
{[\dimA]^{\frac{1}{p}-1} \over \left[\dim_{\tensor{A}} (\pi_1, \pi_2)\right]^{-1/2}} \norm{\unfold_{\pi_1}({\tensor{A})}}_p \leq \norm{\unfold_{\pi_2}({\tensor{A})}}_p \leq  {[\dimA]^{1-\frac{1}{p}} \over \left[\dim_{\tensor{A}} (\pi_2, \pi_1) \right]^{1/2}} \norm{\unfold_{\pi_1}({\tensor{A})}}_p.
\]
\end{enumerate}
\end{thm}

\begin{proof}\
We only prove (a) since (b) follows similarly. For any given $1\leq p\leq 2$, taking $q=2$ in Proposition~\ref{thm:pnorm} implies that the bound between the $l^p$-norm and spectral norm depends only on the total dimension of the tensor, $\dimA = \prod_{n\in[k]} d_n$. Because the total dimension is invariant under any unfolding operation, we have
\begin{equation}\label{eq:volume}
\dimA^{\frac{1}{2}-\frac{1}{p}} \norm{\unfold_{\pi}(\tensor{A})}_\sigma\leq \norm{\unfold_{\pi}(\tensor{A})}_p \leq \norm{\unfold_{\pi}(\tensor{A})}_\sigma, 
\end{equation}
for all $\pi \in \mathcal{P}_{[k]}$.  Combining~\eqref{eq:volume} with Theorem~\ref{thm:main} gives the desired results. 
 \end{proof}

\section{Orthogonal decomposability  and norm equality on upper cones}
\label{sec:OD}

We have seen that the unfolding operation may change the spectral norm by up to a $\poly(d)$ factor for an arbitrary $\tensor{A}$. This is undesirable in many flattening-based algorithms, such as \cite{de2000multilinear,vasilescu2002multilinear}. However, for some specially-structured 
tensors, the operator norm on the partition lattice may not change much either globally or locally.  We demonstrate such a behavior for the following class of tensors:

\begin{defn}[$\pi$-orthogonal decomposable]\label{def:podeco}
Let $\tensor{A}\in \mathbb{R}^{d_1\times\cdots\times d_k}$ be an order-$k$ tensor and consider any partition $\pi\in\mathcal{P}_{[k]}$. Then $\tensor{A}$ is called \emph{$\pi$-orthogonal decomposable}, or $\pi$-OD, over $\mathbb{R}$ if it admits the decomposition
\begin{equation}
\tensor{A}=\sum_{n=1}^r\lambda_n \matr{a}^{(n)}_1 \otimes \cdots \otimes \matr{a}^{(n)}_k,
\label{eq:godeco}
\end{equation}
where $\lambda_1\geq \lambda_2\geq \cdots \geq \lambda_r\geq 0$, and the set of vectors $\{ \matr{a}^{(n)}_i\in\mathbb{R}^{d_i}\colon i\in[k],\ n\in[r]\}$ satisfies 
\begin{equation}
\langle \otimes_{i\in B} \matr{a}^{(n)}_i,\ \otimes_{i\in B} \matr{a}^{(m)}_i \rangle=\delta_{nm},
\end{equation}
for all $B \in \pi$ and all $n, m \in[r]$.
\end{defn}

\medskip
A concept similar to $\pi$-OD, referred to as \emph{biorthogonal eigentensor decomposition}~\cite{yuan2014tensor}, is introduced in the tensor completion literature when $k=3$ and $\pi=\{\{1\}, \{2,3\}\}$. Informally speaking, $\pi$-OD imposes an orthogonality constraint on every block of singular vectors. 

\begin{rmk}[$\zerok$-OD] \label{def:odeco}
When $\pi = \zerok$ in Definition~\ref{def:podeco}, we obtain the special case of $\zerok$-OD tensors, which admit the decomposition \eqref{eq:godeco} while satisfying 
\begin{equation}\label{eq:godeco2}
\langle \matr{a}^{(n)}_i, \ \matr{a}^{(m)}_i \rangle=\delta_{nm},
\end{equation}
for all $i \in[k]$ and all $n,m\in[r]$.
\end{rmk}

The definition of $\zerok$-OD tensors generalizes the definition of orthogonal decomposable tensors presented in~\cite{robeva2014orthogonal}, as we require neither symmetry nor equality of dimension across modes. In fact, a $\zerok$-OD tensor $\tensor{A}$ is a \emph{diagonalizable} tensor \cite{de2000multilinear}, meaning that the core tensor output from higher-order SVD is \emph{superdiagonal} (i.e.,\ entries are zero unless $i_1=\cdots=i_k$).

\begin{lem} \label{pro:odeco}
Consider an order-$k$ tensor $\tensor{A}\in \mathbb{R}^{d_1\times\cdots\times d_k}$.
\begin{enumerate}
\item[(a)]
Let $\pi_1, \pi_2 \in \mathcal{P}_{[k]}$ and $\pi_1 \leq \pi_2$. If $\tensor{A}$ is $\pi_1$-OD, then $\tensor{A}$ is $\pi_2$-OD.
\item[(b)]
Let $\pi\in\mathcal{P}_{[k]}$ and $\pi\neq \onek$. If $\tensor{A}$ is $\pi$-OD, then 
\begin{equation}\label{eq:odecocon}
\norm{\unfold_{\pi}(\tensor{A})}_\sigma=\lambda_1.
\end{equation}
\end{enumerate}
\end{lem}

\begin{proof}\
Part (a): For any two finite sets of vectors $\{\matr{x}_i\}$ and $\{\matr{y}_i\}$ for which $\matr{x}_i, \matr{y}_i \in \mathbb{R}^{d_i},$ we have
\begin{equation}\label{eq:fact}
\langle \otimes_{i\in B}\matr{x}_i,\ \otimes_{j\in B}\matr{y}_j\rangle =\prod_{i\in B} \langle \matr{x}_i,\ \matr{y}_i \rangle, 
\end{equation}
for all $B \subset [k]$.  Suppose $B \in \pi_2$.  If $\pi_1 \leq \pi_2$,  then there exist subsets $C_1, \ldots, C_m \in \pi_1$ such that $C_1 \cup \cdots \cup C_m =  B$.  So,
\[
    \langle \otimes_{i\in B}\matr{x}_i,\ \otimes_{j\in B}\matr{y}_j\rangle = 
    \prod_{i\in B} \langle \matr{x}_i,\ \matr{y}_i \rangle = 
    \prod_{a=1}^m \prod_{i\in C_a} \langle \matr{x}_i,\ \matr{y}_i \rangle = 
    \prod_{a=1}^m \langle \otimes_{i\in C_a}\matr{x}_i,\ \otimes_{j\in C_a}\matr{y}_j\rangle, 
\]
which implies that $\tensor{A}$ is $\pi_2$-OD if $\tensor{A}$ is $\pi_1$-OD.

Part (b):  Suppose $\pi\in\mathcal{P}^\ell_{[k]}$ is of the form $\pi=\{B^\pi_i\colon i\in[\ell]\}$. Note that $\pi\neq \onek$ implies $\ell \geq 2$. Letting $\tau=\{B^\pi_1, (B^\pi_1)^c\}$ where $(B^\pi_1)^c$ denotes the complement of $B^\pi_1$ with respect to $[k]$, we have $\tau \in\mathcal{P}^{2}_{[k]}$ and $\pi\leq \tau$. By Lemma~\ref{pro:odeco}(a), $\tensor{A}$ is $\tau$-OD, so $\tensor{A}$ admits a decomposition of the form
\begin{equation}\label{eq:odeco}
\tensor{A}=\sum_{n=1}^r\lambda_n \matr{a}^{(n)}_1 \otimes \cdots \otimes \matr{a}^{(n)}_k , 
\end{equation}
where
\begin{equation}\label{eq:svd}
\left \langle \otimes_{i \in B_1} \matr{a}^{(n)}_i,\ \otimes_{j\in B_1} \matr{a}^{(m)}_j \right \rangle=\delta_{nm} \quad \text{and} \quad \left \langle \otimes_{i \in (B^\pi_1)^c} \matr{a}^{(n)}_i,\ \otimes_{j\in (B^\pi_1)^c} \matr{a}^{(m)}_j \right \rangle=\delta_{nm}
\end{equation}
for all $n, m\in[r]$.

Now define $\matr{x}_n=\vec(\otimes_{i \in B^\pi_1} \matr{a}^{(n)}_i)$ and $\matr{y}_n=\vec(\otimes_{i \in (B^\pi_1)^c} \matr{a}^{(n)}_i)$ for all $n\in[r]$. By~\eqref{eq:svd}, both $\{\matr{x}_n\}$ and $\{\matr{y}_n\}$ are sets of orthonormal vectors. By the definition of $\unfold_{\tau}(\tensor{A})$, ~\eqref{eq:odeco} implies
\begin{equation}\label{eq:maxlambda1}
\unfold_{\tau}(\tensor{A})=\sum_{n=1}^{r} \lambda_n \matr{x}_n \matr{y}^T_n,
\end{equation}
which is simply the matrix SVD of $\unfold_{\tau}(\tensor{A})$. Hence $\norm{\unfold_\tau(\tensor{A})}_\sigma=\lambda_1$. Using monotonicity (c.f.,\ Proposition~\ref{monotonicity}), we have 
\begin{equation}\label{eq:less}
\norm{\unfold_\pi(\tensor{A})}_\sigma \leq \norm{\unfold_\tau(\tensor{A})}_\sigma =\lambda_1.
\end{equation}
Conversely, by the definition of the spectral norm, we have
\begin{align}\label{eq:maxlambda2}
\norm{\unfold_{\pi}(\tensor{A})}_\sigma &\geq  \left \langle \unfold_{\pi}(\tensor{A}),\ \vec\left(\otimes_{i\in B_1^\pi}\matr{a}^{(1)}_i\right)\otimes\cdots \otimes 
\vec\left(\otimes_{i\in B_\ell^\pi}\matr{a}^{(1)}_i\right) \right \rangle \notag \\
&=\sum_{n=1}^r \lambda_n \Big \langle \vec\left(\otimes_{i\in B^\pi_1}\matr{a}^{(n)}_i\right)\otimes\cdots \otimes 
\vec\left(\otimes_{i\in B^\pi_\ell}\matr{a}^{(n)}_i\right), \notag \\
&\quad  \quad \quad \quad \quad \vec\left(\otimes_{i\in B^\pi_1}\matr{a}^{(1)}_i\right)\otimes\cdots \otimes 
\vec\left(\otimes_{i\in B^\pi_\ell}\matr{a}^{(1)}_i\right) \Big \rangle \notag \\
&=\sum_{n=1}^r \lambda_n \prod_{j\in[\ell]}\left \langle \vec\left(\otimes_{i \in B^\pi_j}\matr{a}^{(n)}_i\right),\ \vec\left(\otimes_{i \in B^\pi_j}\matr{a}_i^{(1)}\right)\right\rangle \notag \\
&=\sum_{n=1}^r \lambda_n \delta_{n,1}  =\lambda_1
\end{align}
where the third line comes from~\eqref{eq:fact} and the last line follows from the fact that $\tensor{A}$ is $\pi$-OD. Combining~\eqref{eq:less} and~\eqref{eq:maxlambda2}, we conclude $\norm{\unfold_{\pi}(\tensor{A})}_\sigma=\lambda_1$.
\end{proof}

\begin{rmk}
The condition $\pi\neq \onek$ in Lemma~\ref{pro:odeco} is needed for \eqref{eq:odecocon} to hold. In fact, consider a $2\times 2$ matrix $\matr{A}=2\matr{e}_1\otimes \matr{e}_1+\matr{e}_1\otimes \matr{e}_2$, where $\{ \matr{e}_i, i\in[2]\}$ is the canonical basis of $\mathbb{R}^2$. Then, $\matr{A}$ is $\onek$-OD with $\lambda_1=2$, but $\norm{\unfold_{\onek}(\tensor{A})}_\sigma=\norm{\tensor{A}}_F=\sqrt{5}\neq 2$.
\end{rmk}

\begin{thm}[Norm equality on upper cones] \label{thm:envelope}
If $\tensor{A}$ is $\pi$-OD, then for any partition in the \emph{upper cone} ${U}_{\pi}=\{\tau\in\mathcal{P}_{[k]}\colon \pi\leq \tau<\onek \}$ of $\pi$, we have
$$
\norm{\unfold_{\tau}(\tensor{A})}_\sigma=\norm{\unfold_\pi(\tensor{A})}_\sigma.
$$
\end{thm}

\begin{proof}
If $\tensor{A}$ is $\pi$-OD, then by Lemma~\ref{pro:odeco}(b), we have $\norm{\unfold_{\pi}(\tensor{A})}_\sigma=\lambda_1$. Given any $\tau\geq \pi$, by Lemma~\ref{pro:odeco}(a), $\tensor{A}$ is also $\tau$-OD. Again, using Lemma~\ref{pro:odeco}(b), we have $\norm{\unfold_{\tau}(\tensor{A})}_\sigma=\lambda_1$. Therefore, $\norm{\unfold_{\tau}(\tensor{A})}_\sigma=\norm{\unfold_\pi(\tensor{A})}_\sigma$.
\end{proof}

Theorem~\ref{thm:envelope} states that the spectral norm is invariant for $\pi$-OD $\tensor{A}$ under any unfolding induced by the partitions in the upper cone ${U}_{\pi}$ of $\pi$. 
This lies in contrast with the $\poly(d)$ factor we have seen for unstructured tensors.

\medskip
\begin{cor} \label{cor:global}
If $\tensor{A}$ is $\zerok$-OD, then for all partitions $\pi\neq \onek$, we have
\[
\norm{\unfold_\pi (\tensor{A})}_\sigma=\norm{\tensor{A}}_\sigma.
\]
\end{cor}

Corollary~\ref{cor:global} implies that for $\zerok$-OD tensors, the operator norm is invariant under any unfolding operations except vectorization.  Lastly, $\pi_1, \pi_2 \in {U}_{\pi_1\wedge\pi_2}$ implies the following corollary:

\begin{cor}
Let $\pi_1, \pi_2\in\mathcal{P}_{[k]}$. If $\tensor{A}$ is $(\pi_1\wedge \pi_2)$-OD, then
\[
\norm{\unfold_{\pi_1}(\tensor{A})}_\sigma=\norm{\unfold_{\pi_2}(\tensor{A})}_\sigma.
\]
\end{cor}

\section{Discussion}
\label{sec:discussion}

In this paper, we presented a new framework representing all possible tensor unfoldings by the partition lattice and established a set of general inequalities quantifying the impact of tensor unfoldings on the operator norms of the resulting tensors. We showed that the comparison bounds scale polynomially in the dimensions $\{d_n\}$ of the tensor, with powers depending on the corresponding partition and block sizes for any pair of tensor unfoldings being compared.
As a direct consequence, we demonstrated how the operator norm of a general tensor is lower and upper bounded by that of its unfoldings.

In general, an unfolding operation may inflate the operator norm by up to a $\poly(d)$ factor, as seen in Corollary~\ref{thm:bottomup}.  Note that the quantity $\dimA$ plays a key role in the worst-case inflation factor and is a manifestation of the curse of dimensionality. Specifically, $\dimA$ can be quite large as the mode dimensions and tensor order increase, with particular sensitivity to the latter.  In such settings, our main result seems to bode poorly for flattening-based algorithms; however, we believe that it should be interpreted with caution because our comparison bounds deal with arbitrary tensors rather than those often sought in applications. In fact, $\pi$-OD tensors permit much tighter bounds in which some unfoldings, including certain matricizations, leave the operator norm relatively unaffected. In practice, $\pi$-OD tensors, or those within a small neighborhood around $\pi$-OD tensors, arise widely in statistical and machine learning applications~\cite{anandkumar2014tensor,kuleshov2015tensor,nickel2011three,wang2016orthogonal}.

Additionally, our work enables us to compare different unfoldings at the same level $\ell$.  Recent work on problems featuring nuclear-norm regularization has shown that not all $n$-mode flattenings are equally preferable~\cite{liu2013tensor}. Indeed, as illustrated in Example~\ref{exam:bound}, the operator norm of level-2 unfoldings (i.e.\ matricizations) can be quite different. Recently, several algorithms have been proposed to account for this behavior. For example, Tomioka \emph{et al.}  \cite{tomioka2010estimation} consider a weighted sum of the norms of all single-mode matricizations. Other techniques include \emph{two-mode matricization}~\cite{wang2016orthogonal}  and \emph{square matricization}~\cite{mu2013square, richard2014statistical}, in which the original tensor is reshaped into a matrix by flattening along multiple modes. Our work provides general bounds to evaluate the effectiveness of such schemes. In particular, the results presented here are used in the theoretical analysis of a two-mode higher-order SVD algorithm proposed recently~\cite{wang2016orthogonal}. 
 
We have not attempted to characterize the degree to which operator norm relations on the partition lattice restrict the original tensor. Essentially, this is a converse problem asking whether $\pi$-OD is a necessary condition for Theorem~\ref{thm:envelope} and Corollary~\ref{cor:global} in addition to being sufficient. If not, it would be useful to determine the extent to which such equalities inform us about the intrinsic structure of the original tensor. From a practical standpoint, norm comparisons between different matricizations are relatively simple, but the optimal manner in which to use this information to learn about the original tensor remains unknown.  

In closing, we emphasize that while this work focuses on theory rather than computational tractability, it possesses practical implications as well. Because direct calculation of the operator norm of a level-$\ell$ tensor is generally computationally prohibitive for $\ell\geq 3$, exploiting level-2 unfoldings may be attractive when the unfolding effect is small enough. Alternatively, for more precise calculations, a number of approximation algorithms exist for higher-order tensor problems~\cite{tomioka2010estimation, yu2014approximate} at the cost of increased computation. Given that the trade-off between accuracy and computation is often unavoidable, our work may be of help in finding an appropriate application-specific balance when working with higher-order tensors.

\section*{Acknowledgements}
This research is supported in part by a Math+X Research Grant from the Simons Foundation,
a Packard Fellowship for Science and Engineering, and an NIH training grant T32-HG000047.

\bibliography{unfoldingDraft}

\begin{bibdiv}
\begin{biblist}

\bib{anandkumar2014tensor}{article}{
      author={Anandkumar, Animashree},
      author={Ge, Rong},
      author={Hsu, Daniel},
      author={Kakade, Sham~M.},
      author={Telgarsky, Matus},
       title={Tensor decompositions for learning latent variable models},
        date={2014},
     journal={The Journal of Machine Learning Research},
      volume={15},
      number={1},
       pages={2773\ndash 2832},
}

\bib{cardoso1990eigen}{inproceedings}{
      author={Cardoso, Jean-FranGois},
       title={Eigen-structure of the fourth-order cumulant tensor with
  application to the blind source separation problem},
organization={IEEE},
        date={1990},
   booktitle={{International Conference on Acoustics, Speech, and Signal
  Processing (ICASSP)}},
       pages={2655\ndash 2658},
}

\bib{de2000multilinear}{article}{
      author={De~Lathauwer, Lieven},
      author={De~Moor, Bart},
      author={Vandewalle, Joos},
       title={A multilinear singular value decomposition},
        date={2000},
     journal={SIAM Journal on Matrix Analysis and Applications},
      volume={21},
      number={4},
       pages={1253\ndash 1278},
}

\bib{friedland2014computational}{article}{
      author={Friedland, Shmuel},
      author={Lim, Lek-Heng},
       title={Computational complexity of tensor nuclear norm},
        date={2016},
     journal={SIAM Journal on Optimization},
      volume={4},
      number={26},
       pages={2378\ndash 2393},
}

\bib{friedland2016computational}{article}{
      author={Friedland, Shmuel},
      author={Lim, Lek-Heng},
       title={Nuclear norm of higher-order tensors},
        date={to appear},
     journal={Mathematics of Computation},
}

\bib{hillar2013most}{article}{
      author={Hillar, Christopher~J.},
      author={Lim, Lek-Heng},
       title={Most tensor problems are {NP}-hard},
        date={2013},
     journal={Journal of the ACM},
      volume={60},
      number={6},
       pages={Article 45},
}

\bib{hoff2015multilinear}{article}{
      author={Hoff, Peter~D.},
       title={Multilinear tensor regression for longitudinal relational data},
        date={2015},
     journal={The Annals of Applied Statistics},
      volume={9},
      number={3},
       pages={1169\ndash 1193},
}

\bib{horn2012matrix}{book}{
      author={Horn, Roger~A.},
      author={Johnson, Charles~R.},
       title={{Matrix Analysis}},
   publisher={Cambridge University Press},
        date={2012},
}

\bib{hu2015relations}{article}{
      author={Hu, Shenglong},
       title={Relations of the nuclear norm of a tensor and its matrix
  flattenings},
        date={2015},
     journal={Linear Algebra and its Applications},
      volume={478},
       pages={188\ndash 199},
}

\bib{kroonenberg1983three}{book}{
      author={Kroonenberg, Pieter~M},
       title={{Three-Mode Principal Component Analysis: Theory and
  Applications}},
   publisher={DSWO Press},
        date={1983},
      volume={2},
}

\bib{kuleshov2015tensor}{inproceedings}{
      author={Kuleshov, V.},
      author={Chaganty, A.},
      author={Liang, P.},
       title={Tensor factorization via matrix factorization},
        date={2015},
   booktitle={{Proceedings of the 18th International Conference on Artificial
  Intelligence and Statistics (AISTATS)}},
       pages={507\ndash 516},
}

\bib{lim2005singular}{inproceedings}{
      author={Lim, Lek-Heng},
       title={Singular values and eigenvalues of tensors: a variational
  approach},
        date={2005},
   booktitle={{Proceedings of the {IEEE} International Workshop on
  Computational Advances in Multi-Sensor Adaptive Processing} ({CAMSAP})},
      volume={1},
       pages={129\ndash 132},
}

\bib{liu2013tensor}{article}{
      author={Liu, Ji},
      author={Musialski, Przemyslaw},
      author={Wonka, Peter},
      author={Ye, Jieping},
       title={Tensor completion for estimating missing values in visual data},
        date={2013},
     journal={{IEEE} Transactions on Pattern Analysis and Machine
  Intelligence},
      volume={35},
      number={1},
       pages={208\ndash 220},
}

\bib{mccullagh1984tensor}{article}{
      author={McCullagh, Peter},
       title={Tensor notation and cumulants of polynomials},
        date={1984},
     journal={Biometrika},
      volume={71},
      number={3},
       pages={461\ndash 476},
}

\bib{mu2013square}{inproceedings}{
      author={Mu, Cun},
      author={Huang, Bo},
      author={Wright, John},
      author={Goldfarb, Donald},
       title={Square deal: Lower bounds and improved relaxations for tensor
  recovery},
        date={2014},
   booktitle={{Proceedings of the 31st International Conference on Machine
  Learning (ICML)}},
      series={JMLR Proceedings},
      volume={32},
   publisher={JMLR.org},
       pages={73\ndash 81},
}

\bib{nickel2011three}{inproceedings}{
      author={Nickel, Maximilian},
      author={Tresp, Volker},
      author={Kriegel, Hans-Peter},
       title={A three-way model for collective learning on multi-relational
  data},
        date={2011},
   booktitle={{Proceedings of the 28th International Conference on Machine
  Learning} ({ICML})},
       pages={809\ndash 816},
}

\bib{omberg2007tensor}{article}{
      author={Omberg, Larsson},
      author={Golub, Gene~H.},
      author={Alter, Orly},
       title={A tensor higher-order singular value decomposition for
  integrative analysis of {DNA} microarray data from different studies},
        date={2007},
     journal={Proceedings of the National Academy of Sciences},
      volume={104},
      number={47},
       pages={18371\ndash 18376},
}

\bib{qi2005eigenvalues}{article}{
      author={Qi, Liqun},
       title={Eigenvalues of a real supersymmetric tensor},
        date={2005},
     journal={Journal of Symbolic Computation},
      volume={40},
      number={6},
       pages={1302\ndash 1324},
}

\bib{ragnarsson2012block}{article}{
      author={Ragnarsson, Stefan},
      author={Van~Loan, Charles~F.},
       title={Block tensor unfoldings},
        date={2012},
     journal={SIAM Journal on Matrix Analysis and Applications},
      volume={33},
      number={1},
       pages={149\ndash 169},
}

\bib{richard2014statistical}{incollection}{
      author={Richard, Emile},
      author={Montanari, Andrea},
       title={A statistical model for tensor {PCA}},
        date={2014},
   booktitle={{Advances in Neural Information Processing Systems} ({NIPS}) 27},
      editor={Ghahramani, Z.},
      editor={Welling, M.},
      editor={Cortes, C.},
      editor={Lawrence, N.~D.},
      editor={Weinberger, K.~Q.},
   publisher={Curran Associates, Inc.},
       pages={2897\ndash 2905},
  url={http://papers.nips.cc/paper/5616-a-statistical-model-for-tensor-pca.pdf},
}

\bib{robeva2014orthogonal}{article}{
      author={Robeva, Elina},
       title={Orthogonal decomposition of symmetric tensors},
        date={2016},
     journal={SIAM Journal on Matrix Analysis and Applications},
      volume={37},
       pages={86\ndash 102},
}

\bib{sun2006beyond}{inproceedings}{
      author={Sun, Jimeng},
      author={Tao, Dacheng},
      author={Faloutsos, Christos},
       title={Beyond streams and graphs: dynamic tensor analysis},
organization={ACM},
        date={2006},
   booktitle={{Proceedings of the 12th {ACM SIGKDD} International Conference on
  Knowledge Discovery and Data Mining}},
       pages={374\ndash 383},
}

\bib{tomioka2010estimation}{article}{
      author={Tomioka, Ryota},
      author={Hayashi, Kohei},
      author={Kashima, Hisashi},
       title={Estimation of low-rank tensors via convex optimization},
        date={2010},
     journal={arXiv:1010.0789},
}

\bib{tomioka2014spectral}{article}{
      author={Tomioka, Ryota},
      author={Suzuki, Taiji},
       title={Spectral norm of random tensors},
        date={2014},
     journal={arXiv:1407.1870},
}

\bib{vasilescu2002multilinear}{incollection}{
      author={Vasilescu, M. Alex~O.},
      author={Terzopoulos, Demetri},
       title={Multilinear analysis of image ensembles: Tensorfaces},
        date={2002},
   booktitle={{Proceedings of European Conference on Computer Vision (ECCV)}},
   publisher={Springer},
       pages={447\ndash 460},
}

\bib{wang2016orthogonal}{article}{
      author={Wang, Miaoyan},
      author={Song, Yun~S.},
       title={Orthogonal tensor decompositions via two-mode higher-order {SVD}
  ({HOSVD})},
        date={2016},
     journal={arXiv:1612.03839},
}

\bib{yu2014approximate}{inproceedings}{
      author={Yu, Yaoliang},
      author={Cheng, Hao},
      author={Zhang, Xinhua},
       title={Approximate low-rank tensor learning},
        date={2014},
   booktitle={{7th {NIPS} Workshop on Optimization for Machine Learning}},
}

\bib{yuan2014tensor}{article}{
      author={Yuan, Ming},
      author={Zhang, Cun-Hui},
       title={On tensor completion via nuclear norm minimization},
        date={2015},
     journal={Foundations of Computational Mathematics},
       pages={1\ndash 38},
}

\bib{zhao2005tricluster}{inproceedings}{
      author={Zhao, Lizhuang},
      author={Zaki, Mohammed~J.},
       title={Tricluster: an effective algorithm for mining coherent clusters
  in {3D} microarray data},
organization={ACM},
        date={2005},
   booktitle={{Proceedings of the 2005 {ACM SIGMOD} International Conference on
  Management of Data}},
       pages={694\ndash 705},
}

\end{biblist}
\end{bibdiv}

\end{document}